\documentclass[a4 paper]{article}
\usepackage[left=30 mm, right=30 mm, top=30 mm, bottom=30 mm]{geometry}

\usepackage{amsmath, amssymb, amsthm, csquotes, hyperref, multirow, float, tikz, subfig, blkarray, enumitem, titlesec}

\titleformat*{\section}{\normalfont\bfseries}
\titleformat*{\subsection}{\normalfont\bfseries}

\newtheorem{theorem}{Theorem}[section]
\newtheorem{lemma}[theorem]{Lemma}
\newtheorem{cor}[theorem]{Corollary}

\newtheorem{prop}[theorem]{Proposition}

\newcommand{\adj}{\operatorname{adj}}
\newcommand{\sign}{\operatorname{sign}}
\newcommand{\sgn}{\operatorname{sgn}}

\usepackage{array}
\newcolumntype{L}[1]{>{\raggedright\let\newline\\\arraybackslash\hspace{0pt}}m{#1}}
\newcolumntype{C}[1]{>{\centering\let\newline\\\arraybackslash\hspace{0pt}}m{#1}}
\newcolumntype{R}[1]{>{\raggedleft\let\newline\\\arraybackslash\hspace{0pt}}m{#1}}

\hypersetup{colorlinks=true, linkcolor=blue, citecolor=blue}
\usetikzlibrary {decorations.markings, arrows.meta, bending, positioning}
\tikzstyle{edge}=[decoration={markings,mark=at position 0.7 with {\arrow{Stealth[scale=1.5]}}},
      postaction={decorate}]

\usepackage{authblk}

\title{Sign pattern matrices associated with cycle graphs that require algebraic positivity}
\author{Sunil Das \footnote{Email address: {\tt sunilmath92@gmail.com}}}
\affil{\small Department of Mathematics, Faculty of Science and Technology (IcfaiTech), ICFAI Foundation for Higher Education, Hyderabad, Telangana - 501203, India}
\date{}

\begin{document}
\maketitle

\begin{abstract}
\noindent A real matrix is said to be positive if its every entry is positive, and a real square matrix $A$ is algebraically positive if there exists a real polynomial $f$ such that $f(A)$ is a positive matrix. A sign pattern matrix $A$ is said to require a property if all matrices having sign pattern as $A$ have that property. In this paper, we characterize all sign pattern matrices associated with cycle graphs that require algebraic positivity. 
\end{abstract}



\section{Introduction}
A real matrix is said to be positive if its every entry is positive. Positive matrices has applications in population models, Markov chains, economic model, low-dimensional topology, epidemiology etc. (see \cite{M00}). A generalization of positive matrices was given by Kirkland, Qiao, and Zhan as algebraically positive matrices. A real square matrix $A$ is algebraically positive if there exists a real polynomial $f$ such that $f(A)$ is a positive matrix. The characterization of algebraically positive matrices was given by the following result.

\begin{theorem}[\cite{KQZ16}]\label{thm1.1}
A real square matrix is algebraically positive if and only if it has a simple real eigenvalue and corresponding left and right positive eigenvectors.
\end{theorem}

We are interested in answering the problem of identifying an algebraically positive matrix from the signs its entries. A sign pattern matrix is a matrix whose every entry is a symbol from the set $\{+,-,0\}$. The qualitative class of a matrix $A$, denoted by $Q(A)$, is defined by the set of all real matrices obtained from $A$ by replacing \enquote*{$+$} by some positive number, \enquote*{$-$} by some negative number and \enquote*{0} by the zero number. A sign pattern matrix $A$ allows a property $P$ if at least one matrix in $Q(A)$ has the property $P$, and requires property $P$ if all matrices in $Q(A)$ have the property $P$. For our convenience, we use the symbols $+_0,-_0,\#$ in sign pattern matrices. Any entry associated $+_0,-_0,\#$ will be considered as a nonnegative, a non-positive, and an arbitrary real number, respectively.  

To represent and analyze sign patterns efficiently, we use concepts of graphs. Let us recall from \cite{BS95, HL14} that the digraph of a sign pattern matrix $A$ of order $n$, denoted by $D(A)$, is defined by a directed graph with vertex set $\{1,2,\ldots,n\}$, where $(\overrightarrow{i,j})$ is an arc of $D(A)$ if and only if $a_{ij}\neq0$. The graph of a sign pattern matrix $A$ of order $n$, denoted by $G(A)$, is defined by an undirected graph with vertex set $\{1,2,\ldots,n\}$, where $[i,j]$ is an edge of $G(A)$ if and only if either $a_{ij}\neq0$ or $a_{ji}\neq0$. 

A directed path $\alpha$ from vertex $u$ to vertex $v$ in $D(A)$ is a sequence $(u=i_0,i_1,\ldots,i_k=v)$ of distinct vertices such that $(\overrightarrow{i_{t-1},i_t})$ is an arc in $D(A)$ for $t=1,2,\ldots,k$. If all vertices are not distinct, then it is called a directed walk. The length of a directed path is the number of arcs lying on the path. Here length of the directed path $\alpha$ is $k$. Sign of the path $\alpha$ is defined by $\sign(\alpha)=\sign(a_{i_0i_1}a_{i_1i_2}\cdots a_{i_{k-1}i_k})$. A directed cycle $\gamma$ of length $k$ or a directed $k$-cycle in $D(A)$ is a sequence $(i_1,i_2,\ldots,i_k,i_1)$ of vertices such that $i_1,i_2,\ldots,i_k$ are distinct and $(\overrightarrow{i_1,i_2}), (\overrightarrow{i_2,i_3}), \ldots, (\overrightarrow{i_k,i_1})$ are arcs in $D(A)$. Sign of the directed cycle $\gamma$ is defined by $\sign(\gamma)=\sign(a_{i_1i_2}a_{i_2i_3}\cdots a_{i_ki_1})$. A square matrix $A$ is said to be irreducible if $D(A)$ is strongly connected, that is, for any ordered pair $(i,j)$ of vertices there is a directed path from $i$ to $j$ in $D(A)$. In $G(A)$, an (undirected) cycle of length $k$ is a sequence $(i_1,i_2,\ldots,i_k,i_1)$ of vertices such that $i_1,i_2,\ldots,i_k$ are distinct and $[i_1,i_2], [i_2,i_3], \ldots, [i_k,i_1]$ are edges in $G(A)$.

Characterization of sign patterns allowing algebraic positivity is still not solved completely, but many results to identify those patterns can be observed in \cite{AP19, D22, DB19, KQZ16}. The problem of identifying sign patterns requiring algebraically positivity is widely open. Path and star sign patterns requiring algebraic positivity are described in \cite{DB19}, and list of all 3-by-3 sign patterns requiring algebraic positivity are presented in \cite{AP19}. The graphs of 3-by-3 tree sign pattern matrices are either stars or paths. For any other 3-by-3 irreducible sign pattern matrix, the underlying undirected graph is a cycle. In this paper, we attempt to generalize the problem of requiring algebraic positivity for cycle graphs. In the process, we also give some results in identifying matrices which are not algebraically positive and whose underlying graph is an arbitrary graph. 

In Section \ref{sec2}, we provide some necessary conditions for a sign pattern matrix to require algebraic positivity. In Section \ref{sec3}, we identify all sign patterns having cycle graphs that require algebraic positivity. Finally, we summerize all main results in Section \ref{sec4}.

Throughout this paper, $\mathbb{R}$ denotes the set of all real numbers, $\mathbb{R}^n$ is the $n$-dimensional Euclidean space, and $A^T$ is the transpose of the matrix $A$. A vector in $\mathbb{R}^n$ is denoted by $\bar{x}$ and its $i$-th entry is denoted by $x_i$. A zero vector is denoted by $\bar{0}$ and an identity matrix is denoted by $I$, whose sizes will be clear from the context. We denote the adjugate of a square matrix $A$ by $\adj(A)$.

\section{Identifying sign patterns that does not require algebraic positivity}\label{sec2}
The following lemma from \cite[pp. 78-80]{HJ13} gives a relation between left and right eigenvectors corresponding to two distinct eigenvalues of a matrix. 

\begin{lemma}[\cite{HJ13}]\label{lem2.1}
Let $A$ be a real square matrix and $\bar{x}, \bar{y} \in \mathbb{R}^n \setminus \{\bar{0}\}$ such that $A\bar{x} = \lambda \bar{x}$ and $\bar{y}^T A = \mu \bar{y}^T$ for some $\lambda,\mu\in\mathbb{R}$. 
\begin{enumerate}
\item If $\lambda\neq\mu$, then $\bar{y}^T\bar{x}=0$.
\item If $\lambda=\mu$ and $\lambda$ has geometric multiplicity 1, then its algebraic multiplicity is 1 if and only if $\bar{y}^T\bar{x}\neq0$.
\end{enumerate}
\end{lemma}

The following lemma shows that any non-negative left or right eigenvector of an algebraically positive matrix must be a scalar multiple of its left or right positive eigenvector, whose existence is given by Theorem \ref{thm1.1}. 

\begin{lemma}\label{lem2.2}
If a matrix $A$ has a non-negative left or right eigenvector $\bar{x}$ with some $x_i=0$, then $A$ is not algebraically positive. 
\end{lemma}

\begin{proof}
Suppose that $A\bar{x}=\lambda \bar{x}$, where $\bar{x}$ is non-negative right eigenvector of $A$ with some $x_i=0$. If $A$ is algebraically positive, then by Theorem \ref{thm1.1}, it has a simple eigenvalue $\mu$ such that $A\bar{z}=\mu \bar{z}$ and $\bar{y}^T A = \mu \bar{y}^T$ for some positive vectors $\bar{y}$ and $\bar{z}$. Since $\bar{x}$ and $\bar{z}$ are independent, and $\mu$ is simple eigenvalue, $\lambda\neq\mu$. But $\bar{y}^T\bar{x}\neq0$, a contradiction to Lemma \ref{lem2.1}.1. Therefore $A$ is not algebraically positive. Similarly, if $A$ has a non-negative left eigenvector $\bar{x}$ with some $x_i=0$, then $A$ is not algebraically positive. 
\end{proof}

Let $A(i,:)$ be the matrix obtained from $A$ by deleting $i$-th row, and $A(:,j)$ be the matrix obtained from $A$ by deleting $j$-th column. The following result gives a sufficient condition to identify sign patterns that do not require algebraic positivity.

\begin{prop}\label{prop2.3}
Let $B\in Q(A)$. If there exist $t,\lambda\in\mathbb{R}$ such that each column of the matrix $(B-\lambda I)(t,:)$ is either a zero vector or has both positive and negative entries, then $A$ does not require algebraic positivity.
\end{prop}

\begin{proof}
Suppose that each column of the matrix $(B-\lambda I)(t,:)$ is either a zero vector or has both positive and negative entries. Keeping $\lambda$ fixed, we can choose the entries of $B$ in such a way that each column sum of $(B-\lambda I)(t,:)$ is zero so that $\bar{x}^TB=\lambda \bar{x}^T$, where
$$x_i=\begin{cases}
1, & \mbox{if } i\neq t; \\
0, & \mbox{if } i=t.
\end{cases}$$ 
So by Lemma \ref{lem2.2}, $B$ is not algebraically positive, and thus $A$ does not require algebraic positivity. 
\end{proof}

If $B\in Q(A)$ and there exist $t,\lambda\in\mathbb{R}$ such that each row of the matrix $(B-\lambda I)(:,t)$ is either a zero vector or has both positive and negative entries, then we can similarly show that $A$ does not require algebraic positivity. Consequently, we have the following result to be used numerously in later part of this paper.

\begin{cor}\label{cor2.4}
Let $A$ be a sign pattern matrix. If there exists $t$ such that the $t$-th column of $A(t,:)$ is either a zero vector or has both $+$ and $-$, and all other columns has a $+$, then $A$ does not require algebraic positivity.
\end{cor}

Let $A$ be a sign pattern matrix, and let the matrices $A_+$, $A_-$ and $B_A$ be defined as follows:
$$(A_+)_{ij}=\begin{cases}
+, & \mbox{if }a_{ij}=+;\\
0, & \mbox{if }a_{ij}\neq+.
\end{cases} \qquad
(A_-)_{ij}=\begin{cases}
-, & \mbox{if }a_{ij}=-;\\
0, & \mbox{if }a_{ij}\neq-.
\end{cases} \qquad B_A=A_+-A_-^T$$
Recall from \cite{D22} that a sign pattern matrix $A$ is AP-irreducible if every row and column of $A$ contains a $+$, and both $A$ and $B_A$ are irreducible. If a sign pattern matrix requires some property, then it also allows that property. So we have the following result from \cite{AP19, KQZ16}.

\begin{lemma}\label{lem2.5}
If a sign pattern matrix requires algebraic positivity, then it is AP-irreducible.
\end{lemma} 

We have the following results from \cite{KQZ16}.

\begin{lemma}[\cite{KQZ16}]\label{lem2.6}
The following statements are equivalent for a square matrix $A$.
\begin{enumerate}
\item $A$ is algebraically positive.
\item $-A$ is algebraically positive.
\item $P^TAP$ is algebraically positive for every permutation matrix $P$.
\item $A-\alpha I$ is algebraically positive for every real number $\alpha$.
\end{enumerate} 
\end{lemma}

A matrix is said to be stable (semi-stable) if all its eigenvalues have negative (non-positive) real part. A sign pattern matrix is potentially stable (semi-stable) if its qualitative class contains a stable (semi-stable) matrix. 

\begin{prop}\label{prop2.7}
Let $A$ be a sign pattern matrix such that it has either a row containing no $-$ or a column containing no $-$. If $A$ is potentially semi-stable, then it does not require algebraic positivity.
\end{prop}

\begin{proof}
Suppose that $A$ requires algebraic positivity. Since $A$ has either a row containing no $-$ or a column containing no $-$, by Theorem \ref{thm1.1}, every matrix in $Q(A)$ has a simple positive eigenvalue. But it contradicts the assumption that $A$ is potentially semi-stable.
\end{proof}

\begin{prop}\label{prop2.8}
A real skew-symmetric matrix of even order cannot be algebraically positive.
\end{prop}

\begin{proof}
Eigenvalues of a real skew-symmetric matrix are either 0 or purely imaginary. So an even order real skew-symmetric matrix cannot have a simple real eigenvalue. Therefore, by Theorem \ref{thm1.1}, a real skew-symmetric matrix of even order cannot be algebraically positive.
\end{proof}

\section{Sign patterns of cycle graphs that require algebraic positivity}\label{sec3}
We shall focus on the sign patterns whose graphs are cycles. If $G(A)$ is a cycle for some sign pattern matrix $A$ of order $n$, we can assume that the cycle is given by the sequence $(1,2,3,\ldots,n,1)$ because of Lemma \ref{lem2.6}.3. So we consider $A$ to be of the form
\begin{equation}\label{cycle}
A=\left[\begin{tabular}{C{15 mm}C{15 mm}C{15 mm}C{15 mm}C{15 mm}C{15 mm}C{15 mm}}
$a_{11}$ & $a_{12}$ & 0 & $\cdots$ & $\cdots$ & 0 & $a_{1n}$ \\[10 mm]
$a_{21}$ & $a_{22}$ & $a_{23}$ & 0 & $\cdots$ & $\cdots$ & 0 \\[10 mm]
0 & $\ddots$ & $\ddots$ & $\ddots$ & $\ddots$ && $\vdots$ \\[10 mm]
$\vdots$ & $\ddots$ & $\ddots$ & $\ddots$ & $\ddots$ & $\ddots$ & $\vdots$ \\[10 mm]
$\vdots$ && $\ddots$ & $\ddots$ & $\ddots$ & $\ddots$ & 0 \\[10 mm]
0 & $\cdots$ & $\cdots$ & 0 & $a_{n-2,n-1}$ & $a_{n-1,n-1}$ & $a_{n-1,n}$ \\[10 mm]
$a_{n1}$ & 0 & $\cdots$ & $\cdots$ & 0 & $a_{n,n-1}$ & $a_{nn}$ 
\end{tabular}\right],
\end{equation}
where $a_{ij}\in\{+,-,0\}$ for all $i,j$, and either $a_{ij}\neq0$ or $a_{ji}\neq0$ for $j=i+1 \mod n$.

\begin{lemma}\label{lem3.1}
Let $A$ be a sign pattern matrix of the form \eqref{cycle}. If $A$ requires algebraic positivity, then we have the following:
\begin{enumerate}
\item There do not exist $i,j$ such that $a_{i,i+1}=0$ and $a_{j+1,j}=0$.
\item There do not exist $i,j$ such that $a_{i,i+1}=-_0$, $a_{i+1,i}=+_0$ and $a_{j+1,j}=-_0$, $a_{j,j+1}=+_0$.
\end{enumerate}
\end{lemma} 

\begin{proof} 
Since $A$ requires algebraic positivity, by Lemma \ref{lem2.5}, $A$ is AP-irreducible. Let us say that $(\overrightarrow{i,j})$ is a forward arc if $j=i+1\mod n$, and $(\overrightarrow{j,i})$ is a backward arc if $j=i+1\mod n$.
\begin{enumerate}
\item Suppose that there exist $i,j$ such that $a_{i,i+1}=0$ and $a_{j+1,j}=0$. Since $A$ is of the form \eqref{cycle}, $i\neq j$. Now any directed path in $D(A)$ to vertex $j$ can be traversed via forward arcs only, and any directed path in $D(A)$ from vertex $i$ can be traversed via backward arcs only. So there is no directed path from $i$ to $j$ in $D(A)$, and thus $A$ is reducible. So $A$ is not AP-irreducible, a contradiction.
\item Suppose that there exist $i,j$ such that $a_{i,i+1}=-_0$, $a_{i+1,i}=+_0$ and $a_{j+1,j}=-_0$, $a_{j,j+1}=+_0$. So $(B_A)_{i,i+1}=0$ and $(B_A)_{j+1,j}=0$. Now any directed path in $D(B_A)$ to vertex $j$ can be traversed via forward arcs only, and any directed path in $D(B_A)$ from vertex $i$ can be traversed via backward arcs only. So there is no directed path from $i$ to $j$ in $D(B_A)$, and thus $B_A$ is reducible. So $A$ is not AP-irreducible, a contradiction.
\end{enumerate}
\end{proof}

The following result shows that for a sign pattern matrix $A$ with cycle graph to require algebraic positivity, we must have a directed $n$-cycle in $D(A)$ either with positive arcs only or with negative arcs only.

\begin{lemma}\label{lem3.2}
Let $A$ be a sign pattern matrix of order $n$ of the form \eqref{cycle}. If $A$ requires algebraic positivity, then either $D(A_+)$ or $D(A_-)$ contains a directed cycle of length $n$. 
\end{lemma}

\begin{proof}
Since $A$ requires algebraic positivity, due to Lemma \ref{lem2.5}, we can assume that every row and column of $A$ contains a $+$. 

Suppose that neither $D(A_+)$ nor $D(A_-)$ contains a directed cycle of length $n$. We first prove the following three statements successively.
\begin{enumerate}
\item $a_{ij}=a_{ji}$ for all $i,j$. 
\item There exist no three distinct vertices $i,j,k$ such that $a_{ij}a_{jk}=+$.
\item $a_{ii}=0$ for all $i$.
\end{enumerate}
Our goal is to arrive at a contradiction with the help of these three statements. Since $A$ is of the form \eqref{cycle}, and neither $D(A_+)$ nor $D(A_-)$ contains a directed cycle of length $n$, we can assume that there exist $i,j\in\{1,2,\ldots,n-1\}$ such that $a_{i,i+1}=+$ and $a_{j,j+1}=-$. 
\begin{enumerate}
\item Suppose that $a_{i+1,i}=-_0$. Because of Lemma \ref{lem2.6}.3, we can consider that $a_{12}=+$, and $a_{k+1,k}=-_0$ for some $k<t=\min\{i:a_{i,i+1}=-_0\}$. Then $a_{k,k+1}=+$, and thus by Lemma \ref{lem3.1}.2, $a_{t+1,t}=-$. 

Since every column of $A$ contains a $+$, and $a_{k+1,k}=-_0$ for some $k<t=\min\{i:a_{i,i+1}=-_0\}$, the $t$-th column of $A(t,:)$ has both $+$ and $-$, and all other columns of $A(t,:)$ has a $+$. So by Corollary \ref{cor2.4}, $A$ does not require algebraic positivity, a contradiction. Thus $a_{i+1,i}=+$.

Therefore by Lemma \ref{lem2.6}, we can conclude that $a_{ij}=+$ implies $a_{ji}=+$, and $a_{ij}=-$ implies $a_{ji}=-$. 
\item Since $A$ has the form \eqref{cycle}, $a_{ij}\neq0$ whenever $|i-j|=1$. Suppose that there exist three distinct vertices $i,j,k$ such that $a_{ij}a_{jk}=+$. Because of Lemma \ref{lem2.6}, we can assume that $a_{12},a_{23}=+$ and $a_{34}=-$. So $a_{43}=-$.

Since every column of $A$ contains a $+$, the third column of $A(3,:)$ has a $+$ and a $-$, and all other columns has a $+$. So by Corollary \ref{cor2.4}, $A$ does not require algebraic positivity, a contradiction. So there exist no three distinct vertices $i,j,k$ such that $a_{ij}a_{jk}=+$. Consequently, $n$ is even.
\item We can assume that $a_{12}=+$ and $a_{23}=-$. Then $a_{21}=+$ and $a_{32}=-$.
\begin{enumerate}
\item If $a_{11}=+$, then the second column of $A(2,:)$ has both $+$ and $-$, and all other columns has a $+$, because every column of $A$ contains a $+$. So by Corollary \ref{cor2.4}, $A$ does not require algebraic positivity, a contradiction.
\item If $a_{22}=+$, then the first column of $A(1,:)$ has both $+$ and $-$, and all other columns has a $+$, because every column of $A$ contains a $+$. So by Corollary \ref{cor2.4}, $A$ does not require algebraic positivity, a contradiction. 
\end{enumerate}
Thus $a_{ii}\neq+$ for all $i$. Again by Lemma \ref{lem2.6}, we have $a_{ii}\neq-$ for all $i$. Therefore $a_{ii}=0$ for all $i$.
\end{enumerate}
Using the above three statements, we can assume that $n$ is even, $a_{1n}=a_{n1}=-$, $a_{i,i+1}=a_{i+1,i}=+$ for all odd $i$, and $a_{i,i+1}=a_{i+1,i}=-$ for all even $i$. Consider $B\in Q(A)$ by replacing $+$ with 1 and $-$ with $-1$. Then $B\bar{x}=0$, where 
$$x_i=\begin{cases}
1, & \mbox{if } i \mbox{ is odd}; \\
0, & \mbox{if } i \mbox{ is even}.
\end{cases}$$ 
So by Lemma \ref{lem2.2}, $B$ is not algebraically positive. Therefore $A$ does not require algebraic positivity, a contradiction. Hence either $D(A_+)$ or $D(A_-)$ contains a directed cycle of length $n$.
\end{proof}

The following two results successively shows that for a sign pattern matrix $A$ to require algebraic positivity, no vertex is covered by both positive and negative directed 2-cycles in $D(A)$, and $a_{ii}=-_0$ if vertex $i$ is covered by a negative directed 2-cycle in $D(A)$.

\begin{lemma}\label{lem3.3}
Let $A$ be a sign pattern matrix of order $n$ such that it has the form \eqref{cycle}. If $D(A_+)$ contains a directed cycle of length $n$, and $A$ requires algebraic positivity, then no vertex is covered by both positive and negative directed 2-cycles in $D(A)$.
\end{lemma}

\begin{proof}
Let us assume that $A$ requires algebraic positivity. Since $D(A_+)$ contains a directed cycle of length $n$, assume that $a_{ij}=+$ for all $i,j$ with $j=i+1\mod n$. Suppose that we have some vertex covered by both positive and negative directed 2-cycles. Then we can assume that either $a_{1n}=+$ and $a_{21}=-$ or $a_{21}=-$ and $a_{32}=+$, because of Lemma \ref{lem2.6}. 
\begin{enumerate}
\item Assume that $a_{1n}=+$ and $a_{21}=-$. Now the second row of $A(:,2)$ has both $+$ and $-$, and all other rows have a $+$. So by Corollary \ref{cor2.4}, $A$ does not require algebraic positivity, a contradiction.
\item Assume that $a_{21}=-$ and $a_{32}=+$. Now the first column of $A(1,:)$ has both $+$ and $-$, and all other columns have a $+$. So by Corollary \ref{cor2.4}, $A$ does not require algebraic positivity, a contradiction.
\end{enumerate}
\end{proof}

\begin{lemma}\label{lem3.4}
Let $A$ be a sign pattern matrix of order $n$ having the form \eqref{cycle}, and $D(A_+)$ contains a directed cycle of length $n$. If $a_{ij}a_{ji}=-$ for some $i\neq j$ with $a_{ii}=+$, then $A$ does not require algebraic positivity.
\end{lemma}

\begin{proof}
Since $D(A_+)$ contains a directed cycle of length $n$, assume that $a_{ij}=+$ for all $i,j$ with $j=i+1\mod n$. Due to Lemma \ref{lem2.6}, we can consider that $a_{21}=-$ and $a_{22}=+$. Now the first column of $A(1,:)$ has both $+$ and $-$, and all other columns has a $+$. So by Corollary \ref{cor2.4}, $A$ does not require algebraic positivity.
\end{proof}

A square sign pattern matrix $A$ is said to be sign non-singular (SNS) if every matrix in $Q(A)$ is non-singular. For a square matrix $A$ of order $n$, the standard determinant expansion of $A$ is given by
\begin{equation}
\det A=\sum_{\sigma\in S_n} \sgn(\sigma) a_{1\sigma(1)}a_{2\sigma(2)}\cdots a_{n\sigma(n)},
\end{equation}
where the sum is taken over the set $S_n$ of all permutations of the set $\{1,2,\ldots,n\}$, and $\sgn(\sigma)$ is the sign of the permutation $\sigma$.

\begin{lemma}[\cite{BS95}]\label{lem3.5}
Let $A=[a_{ij}]$ be a matrix of order $n$. Then the following are equivalent:
\begin{enumerate}
\item $A$ is an SNS matrix.
\item There is a nonzero term in the standard determinant expansion of $A$ and every nonzero term has the same sign.
\end{enumerate}
\end{lemma}

Given an SNS matrix $A$ with negative main diagonal, Theorem 3.2.5 in \cite{BS95} provides a necessary and sufficient condition for a non-diagonal entry of $B^{-1}$ to have the same sign for all $B\in Q(A)$. We generalize this result as follows.

\begin{theorem}\label{thm3.6}
Let $A=[a_{ij}]$ be an SNS matrix of order $n$, and let $r$ and $s$ be distinct integers with $1\leq r,s\leq n$. The entries in the $(s,r)$-position of matrices in $\{\adj(B):B\in Q(A)\}$ are of the same sign if and only if for every directed path $\alpha$ from $s$ to $r$ in the signed digraph $D(A)$, the sign $(-)^l\sign(\alpha)\cdot\sign(\beta)$ is the same for all nonzero term $\beta$ in the standard determinant expansion of the matrix $A(\alpha)$ obtained from $A$ by deleting rows and columns corresponding to the vertices covered by $\alpha$, where $l$ is the length of $\alpha$.
\end{theorem}

\begin{proof}
Given a matrix $B\in Q(A)$, the entry in the $(s,r)$-position of $\adj(B)$ is $(-1)^{r+s}$ times the determinant of the matrix obtained from $B$ by deleting the $r$-th row and the $s$-th column. Let $\bar{e}_r$ be the $n\times 1$ column vector having only one nonzero entry 1 at position $r$. So the entry in the $(s,r)$-position of $\adj(B)$ is $\det B(s\leftarrow \bar{e}_r)$, where the matrix $\tilde{B}=B(s\leftarrow \bar{e}_r)$ obtained from $B$ by replacing $s$-th column of $B$ with $\bar{e}_r$. Then $\tilde{b}_{rs}=1$. Every nonzero term in $\det B(s\leftarrow \bar{e}_r)$ can be expressed as 
$$(-1)^l(\tilde{b}_{rs}b_{si_1}b_{i_1i_2}\cdots b_{i_kr})\beta=(-1)^l(b_{si_1}b_{i_1i_2}\cdots b_{i_kr})\beta,$$
where $\alpha=(s,i_1,i_2,\ldots,i_k,r)$ is a directed path in $D(A)$, $l$ is the length of $\alpha$, and $\beta$ is a nonzero term in the determinant expansion of $B(\alpha)$ obtained from $B$ by deleting rows and columns corresponding to the vertices covered by $\alpha$. Since $A$ is SNS, $A(\alpha)$ is SNS provided that $\sign(\alpha)\neq0$. Therefore the result follows from Lemma \ref{lem3.5}.
\end{proof}

The following three lemmas from \cite{DB19,KQZ16} provide some sufficient conditions for any sign pattern matrix to require algebraic positivity.

\begin{lemma}[\cite{KQZ16}]\label{lem3.7}
If $A$ is a real matrix and there is a positive integer $k$ such that $A^k$ is algebraically positive, then $A$ is algebraically positive.
\end{lemma}

\begin{lemma}[\cite{DB19}]\label{lem3.8}
If $\adj(A)$ is algebraically positive for some real square matrix $A$, then $A$ is algebraically positive.
\end{lemma}

\begin{lemma}[\cite{DB19}]\label{lem3.9}
If all nonzero off-diagonal entries of an irreducible sign pattern matrix are the same sign ($+$ or $-$), then that sign pattern matrix requires algebraic positivity.
\end{lemma}

\subsection{Cycles with loops and negative 2-cycles only}
The following lemma gives a necessary condition for a sign pattern matrix associated with cycle graphs to require algebraic positivity if no diagonal entries are positive and no 2-cycles are positive.

\begin{lemma}\label{lem3.11}
Let $A$ be a sign pattern matrix of order $n$ of the form \eqref{cycle} such that $a_{ii}=-_0$ for all $i$ and all directed 2-cycles in $D(A)$ are negative. Let $D(A_+)$ contains a directed $n$-cycle, $D(A_-)$ does not contain a directed $n$-cycle and $a_{ij}a_{ji}=-$ for some $i\neq j$. If $A$ is not SNS and $A$ requires algebraic positivity, then $a_{ii}=-$ and $a_{jj}=-$ whenever $a_{ij}=+$ and $a_{ji}=0$ for some $i\neq j$.
\end{lemma}

\begin{proof}
Since $D(A_+)$ contains a directed $n$-cycle, let $a_{ij}=+$ for all $i,j$ with $j=i+1\mod n$. Since $D(A_-)$ does not contain a directed $n$-cycle, assume that $a_{1n}=0$. 

Now $D(A)$ has no directed $k$-cycle for $2<k<n$ and only one directed $n$-cycle with sign $+$. The standard determinant expansion of every matrix in $Q(A)$ has a term with sign $(-)^{n-1}$ corresponding to the $n$-cycle. Since $A$ is not SNS, the standard determinant expansion of every matrix in $Q(A)$ has a term with sign $(-)^n$. So we can choose $B\in Q(A)$ such that $\det B=(-1)^n$. Moreover, all 2-cycles in $D(A)$ are negative. So the coefficients of $x^i$ in the characteristic polynomial of $B$ are non-negative for $i=1,2,\ldots,n$. Consequently, the characteristic equation $p(B)=0$ has no positive real root, by Descartes' rule of signs. Thus all real eigenvalues of $B$ are negative, as $\det B\neq0$. Since $A$ requires algebraic positivity, by Theorem \ref{thm1.1}, $B$ has positive left and right eigenvectors corresponding to a negative eigenvalue.

If $a_{nn}=0$, then for any $\lambda<0$, the $n$-th column of $B-\lambda I$ is a nonzero, non-negative vector so that $B-\lambda I$ does not have a positive left eigenvector. Similarly, if $a_{11}=0$, then for any $\lambda<0$, the first row of $B-\lambda I$ is a nonzero, non-negative vector so that $B-\lambda I$ does not have a positive right eigenvector. Thus if either $a_{11}=0$ or $a_{nn}=0$, then $B$ is not algebraically positive so that $A$ does not require algebraic positivity, a contradiction. Therefore $a_{11}=-$ and $a_{nn}=-$. Hence the result follows.
\end{proof}

The following result is a necessary condition for a sign pattern of the form \eqref{cycle} to require algebraic positivity when there is no positive directed 2-cycle in its digraph.

\begin{lemma}\label{lem3.12}
Let $A$ be a sign pattern matrix of order $n$ of the form \eqref{cycle} and all directed 2-cycles in $D(A)$ are negative. Let $D(A_+)$ contains a directed $n$-cycle, $a_{ij}a_{ji}=-$ for some $i\neq j$, and $A$ requires algebraic positivity. 
\begin{enumerate}
\item If $D(A_-)$ does not contain an $n$-cycle and $a_{ii}=-_0$ for all $i$, then $A$ is SNS.
\item If $D(A_-)$ contains an $n$-cycle, then $n$ is odd and $a_{ii}=0$ for all $i$. 
\end{enumerate} 
\end{lemma}

\begin{proof}
Assume that $A$ requires algebraic positivity. Since $D(A_+)$ contains a directed $n$-cycle, let $a_{ij}=+$ for all $i,j$ with $j=i+1\mod n$.
\begin{enumerate}
\item Suppose that $A$ is not SNS. Since $a_{ij}a_{ji}=-$ for some $i\neq j$ and $D(A_-)$ does not contain a directed $n$-cycle, assume that $a_{21}=-$ and $a_{1n}=0$. Further, by Lemma \ref{lem3.11}, $a_{ii}=-$ and $a_{jj}=-$ provided that $a_{ji}=0$ for some $j=i+1\mod n$. So for every $k$-th column of $A$, either $a_{kk}=-$ or $a_{k+1,k}=-$ for $1\leq k\leq n-1$, and $a_{nn}=-$.

Now we can choose $B\in Q(A)$ and $\lambda<0$ such that the second column of $(B-\lambda I)(1,:)$ is zero vector or has both $+$ and $-$, and all other columns of $(B-\lambda I)(t,:)$ has both $+$ and $-$. So by Proposition \ref{prop2.3}, $A$ does not require algebraic positivity, a contradiction. Therefore $A$ is SNS.
\item Suppose that $D(A_-)$ contains an $n$-cycle. Then $a_{ij}a_{ji}=-$ for all $i,j$ with $j=i+1\mod n$. So by Lemma \ref{lem2.6} and \ref{lem3.4}, $a_{ii}=0$ for all $i$. If $n$ is even, we can choose $B\in Q(A)$ by replacing $+$ with 1 and $-$ by $-1$ so that $B$ is skew-symmetric. Then by Proposition \ref{prop2.8}, $B$ is not algebraically positive, a contradiction to the assumption that $A$ requires algebraic positivity. Therefore $n$ is odd. 
\end{enumerate}
\end{proof}

The following lemma provides a sufficient condition a sign pattern matrix to require algebraic positivity using Theorem \ref{thm3.6}.

\begin{lemma}\label{lem3.13}
Let $A$ be a sign pattern matrix of order $n$ of the form \eqref{cycle}, all directed 2-cycles in $D(A)$ are negative and $a_{ij}=+$ for all $i,j$ with $j=i+1\mod n$. If $A$ is an SNS matrix such that $a_{ii}=-_0$ for all $i$, then $A$ requires algebraic positivity. 
\end{lemma}

\begin{proof}
Let $B\in Q(A)$. Since all directed 2-cycles in $D(A)$ are negative and $a_{ij}=+$ for all $i,j$ with $j=i+1\mod n$, every path in $D(A)$ consists of either all positive edges or all negative edges.

Given two distinct indices $r$ and $s$, we wish to determine the sign of the $(s,r)$-entry of $\adj(B)$, that is, $(-)^l\sign(\alpha)\cdot\sign(\beta)$, where $\alpha=(s,i_1,i_2,\ldots,i_k,r)$ is a directed path from $s$ to $r$ in $D(A)$, $l$ is the length of $\alpha$, and $\beta$ is a nonzero term in the determinant expansion of $A(\alpha)$ obtained from $A$ by deleting rows and columns corresponding to the vertices covered by $\alpha$. 
\begin{enumerate}
\item Let $n$ be odd. Since $A$ is SNS, and $D(A)$ contains the cycle $\gamma=(1,2,3,\ldots,n,1)$ with $\sign\gamma=+$, by Lemma \ref{lem3.5}, every nonzero term in the determinant expansion of $B$ has sign $(-)^{n-1}\cdot+=+$.
\begin{enumerate}
\item Assume that the directed path $\alpha$ consists of all positive edges.
\begin{enumerate}
\item If $l$ is even, total number of vertices covered in $\alpha$ is odd. So $B(\alpha)$ is of even order, and thus $\sign(\beta)=+_0$. Moreover, $\sign(\alpha)=+$. In this case, the sign of the contributing term to the $(s,r)$-entry of $\adj(B)$ is $(-)^l\sign(\alpha)\cdot\sign(\beta)=+\cdot+\cdot+_0=+_0$. In particular, if $l=n-1$, then $\alpha$ covers all vertices. So the sign of the contributing term to the $(s,r)$-entry of $\adj(B)$ is $(-)^l\sign(\alpha)=+\cdot+=+$. Consequently, the $(s,r)$-entry of $\adj(B)$ has a contributing term with sign $+$ for all $(s,r)\in\{(2,1),(3,2),\ldots,(n,n-1),(1,n)\}$.
\item If $l$ is odd, total number of vertices covered in $\alpha$ is even. So $B(\alpha)$ is of odd order, and thus $\sign(\beta)=-_0$. Moreover, $\sign(\alpha)=+$. In this case, the sign of the contributing term to the $(s,r)$-entry of $\adj(B)$ is $(-)^l\sign(\alpha)\cdot\sign(\beta)=-\cdot+\cdot-_0=+_0$.
\end{enumerate}
\item Assume that the directed path $\alpha$ consists of all negative edges. 
\begin{enumerate}
\item If $l$ is even, total number of vertices covered in $\alpha$ is odd. So $B(\alpha)$ is of even order, and thus $\sign(\beta)=+_0$. Moreover, $\sign(\alpha)=+$. In this case, the sign of the contributing term to the $(s,r)$-entry of $\adj(B)$ is $(-)^l\sign(\alpha)\cdot\sign(\beta)=+\cdot+\cdot+_0=+_0$. 
\item If $l$ is odd, total number of vertices covered in $\alpha$ is even. So $\sign(\alpha)=+$. Let $B[\alpha]$ be the submatrix of $B$ having rows and columns corresponding to the indices in $\alpha$. Since $\alpha$ consists of only negative arcs and covers even number of vertices, $B[\alpha]$ has a nonzero term $\delta$ in its determinant expansion with sign $+$. Now $B(\alpha)$ is of odd order. If $\beta$ is a nonzero term in the determinant expansion of $B(\alpha)$, then $\sign(\beta)=-$ so that $\det B$ has a term with sign $\sign(\delta)\cdot\sign(\beta)=-$, a contradiction. So $\sign(\beta)=0$. Thus the sign of the contributing term to the $(s,r)$-entry of $\adj(B)$ is $(-)^l\sign(\alpha)\cdot\sign(\beta)=-\cdot-\cdot 0=0$. 
\end{enumerate}
\end{enumerate}
Thus all nonzero off-diagonal entries of $\adj(B)$ have the sign $+$. Further, $\adj(B)$ is irreducible because the $(s,r)$-entry of $\adj(B)$ has the sign $+$ for all $(s,r)\in\{(2,1),(3,2),\ldots,(n,n-1),(1,n)\}$. Therefore, by Lemma \ref{lem3.8}, $B$ is algebraically positive. Hence $A$ requires algebraic positivity.
\item Let $n$ be even. Since $A$ is SNS, and $D(A)$ contains the cycle $\gamma=(1,2,3,\ldots,n,1)$ with $\sign\gamma=+$, by Lemma \ref{lem3.5}, every nonzero term in the determinant expansion of $B$ has sign $(-)^{n-1}\cdot+=-$.
\begin{enumerate}
\item Assume that the directed path $\alpha$ consists of all positive edges.
\begin{enumerate}
\item If $l$ is even, total number of vertices covered in $\alpha$ is odd. So $B(\alpha)$ is of odd order, and thus $\sign(\beta)=-_0$. Moreover, $\sign(\alpha)=+$. In this case, the sign of the contributing term to the $(s,r)$-entry of $\adj(B)$ is $(-)^l\sign(\alpha)\cdot\sign(\beta)=+\cdot+\cdot-_0=-_0$. 
\item If $l$ is odd, total number of vertices covered in $\alpha$ is even. So $B(\alpha)$ is of even order, and thus $\sign(\beta)=+_0$. Moreover, $\sign(\alpha)=+$. In this case, the sign of the contributing term to the $(s,r)$-entry of $\adj(B)$ is $(-)^l\sign(\alpha)\cdot\sign(\beta)=-\cdot+\cdot+_0=-_0$. In particular, if $l=n-1$, then $\alpha$ covers all vertices. So the sign of the contributing term to the $(s,r)$-entry of $\adj(B)$ is $(-)^l\sign(\alpha)=-\cdot+=-$. Consequently, the $(s,r)$-entry of $\adj(B)$ has a contributing term with sign $-$ for all $(s,r)\in\{(2,1),(3,2),\ldots,(n,n-1),(1,n)\}$.
\end{enumerate}
\item Assume that the directed path $\alpha$ consists of all negative edges. 
\begin{enumerate}
\item If $l$ is even, total number of vertices covered in $\alpha$ is odd. So $B(\alpha)$ is of odd order, and thus $\sign(\beta)=-_0$. Moreover, $\sign(\alpha)=+$. In this case, the sign of the contributing term to the $(s,r)$-entry of $\adj(B)$ is $(-)^l\sign(\alpha)\cdot\sign(\beta)=+\cdot+\cdot-_0=-_0$. 
\item If $l$ is odd, total number of vertices covered in $\alpha$ is even. So $\sign(\alpha)=+$. Let $B[\alpha]$ be the submatrix of $B$ having rows and columns corresponding to the indices in $\alpha$. Since $\alpha$ consists of only negative arcs and covers even number of vertices, $B[\alpha]$ has a nonzero term $\delta$ in its determinant expansion with sign $+$. Now $B(\alpha)$ is of even order. If $\beta$ is a nonzero term in the determinant expansion of $B(\alpha)$, then $\sign(\beta)=+$ so that $\det B$ has a term with sign $\sign(\delta)\cdot\sign(\beta)=+$, a contradiction. So $\sign(\beta)=0$. Thus the sign of the contributing term to the $(s,r)$-entry of $\adj(B)$ is $(-)^l\sign(\alpha)\cdot\sign(\beta)=-\cdot-\cdot 0=0$. 
\end{enumerate}
\end{enumerate}
Thus all nonzero off-diagonal entries of $\adj(B)$ have the sign $-$. Further, $\adj(B)$ is irreducible because the $(s,r)$-entry of $\adj(B)$ has the sign $-$ for all $(s,r)\in\{(2,1),(3,2),\ldots,(n,n-1),(1,n)\}$. Therefore, by Lemma \ref{lem3.8}, $B$ is algebraically positive. Hence $A$ requires algebraic positivity.
\end{enumerate}
\end{proof}

\begin{cor}\label{cor3.14}
Let $A$ be a sign pattern matrix of order $n$ of the form \eqref{cycle}, all directed 2-cycles in $D(A)$ are negative and $a_{ij}=+$ for all $i,j$ with $j=i+1\mod n$. If $a_{ii}=-_0$ for all $i$ covered by negative directed 2-cycles in $D(A)$ and $a_{ii}=+$ for some $i$, then $A$ requires algebraic positivity. 
\end{cor}

\begin{proof}
Let $B\in Q(A)$, and let $a_{ii}=+$ for some vertex $i$ covered by no negative directed 2-cycles in $D(A)$. If $k=\arg\max_ib_{ii}$, then $k$ is covered by no negative directed 2-cycles in $D(A)$. Consider the pattern of $B-b_{kk}I$, which is SNS and all its diagonal entries are from $\{-,0\}$. Then by Lemma \ref{lem3.13}, $B-b_{kk}I$ is algebraically positive. So by Lemma \ref{lem2.6}, $B$ is algebraically positive, and hence $A$ requires algebraic positivity.
\end{proof}

We characterize all sign pattern matrices $A$ that require algebraic positivity, where $G(A)$ is a cycle and all directed 2-cycles in $D(A)$ are negative, through the following result.

\begin{theorem}\label{thm3.15}
Let $A$ be an irreducible sign pattern matrix of order $n$ of the form \eqref{cycle} such that $a_{ii}\neq0$ for some $i$. If all directed 2-cycles in $D(A)$ are negative, then $A$ requires algebraic positivity if and only if either $A$ or $-A$ satisfies one of the following conditions:
\begin{enumerate}
\item All nonzero off-diagonal entries have the sign $+$.
\item $D(A_+)$ contains a directed $n$-cycle and $D(A_-)$ does not contain a directed $n$-cycle such that $A$ is SNS and $a_{ii}=-_0$ for all $i$.
\item $D(A_+)$ contains a directed $n$-cycle and $D(A_-)$ does not contain a directed $n$-cycle such that $a_{ii}=+_0$ for all vertices $i$ not covered by any directed negative 2-cycles in $D(A)$.
\end{enumerate}
\end{theorem}

\begin{proof}
Suppose that $A$ requires algebraic positivity. Then by Lemma \ref{lem3.2}, either $D(A_+)$ or $D(A_-)$ contains a directed $n$-cycle. Using Lemma \ref{lem2.6}, we can assume that $D(A_+)$ contains a directed $n$-cycle.

Suppose that $A$ has $-$ as a off-diagonal entry. Since $a_{ij}a_{ji}=-_0$ for all $i\neq j$, by Lemma \ref{lem3.4}, no vertex $k$ with $a_{kk}=+$ is covered by any directed negative 2-cycles in $D(A)$. Since $a_{ii}\neq0$ for some $i$, by Lemma \ref{lem3.12}, $D(A_-)$ does not contain a directed $n$-cycle. If $a_{ii}=-_0$ for all $i$, then $A$ is SNS. Therefore conditions 1, 2, 3 are necessary for $A$ to require algebraic positivity.

Lemma \ref{lem3.9}, \ref{lem3.13} and Corollary \ref{cor3.14} respectively show that conditions 1, 2, 3 are sufficient for $A$ to require algebraic positivity.
\end{proof}

\subsection{Cycles without loops and negative 2-cycles only}
The following result characterizes all sign pattern matrices of the form \eqref{cycle} with $a_{ii}=0$ for all $i$ that require algebraic positivity.

\begin{theorem}\label{thm3.10}
Let $A$ be a sign pattern matrix of order $n$ of the form \eqref{cycle} such that $a_{ii}=0$ for all $i$ and all directed 2-cycles are negative. Then $A$ requires algebraic positivity if and only if  either $D(A_+)$ or $D(A_-)$ contains a directed $n$-cycle, and one of the following conditions hold:
\begin{enumerate}
\item $n$ is odd.
\item $n$ is even and $A$ is SNS.
\end{enumerate} 
\end{theorem}

\begin{proof}
Suppose that $A$ requires algebraic positivity. Then by Lemma \ref{lem3.2}, either $D(A_+)$ or $D(A_-)$ contains a directed cycle of length $n$. 

Suppose that $n$ is even and $A$ is not SNS. Since all directed 2-cycles are negative and $a_{ii}=0$ for all $i$, we can choose $B\in Q(A)$ in such a way that the characteristic polynomial $p(B)$ is an even degree polynomial such that coefficient of $x^i$ is positive if $i$ is even number, and coefficient of $x^i$ is zero if $i$ is odd number. In that case, $p(B)=0$ has no real root. So $B$ is not algebraically positive, a contradiction.

Conversely, without loss of generality assume that $D(A_+)$ contains a directed $n$-cycle. So we can assume that $a_{ij}=+$ for $j=i+1 \mod n$. 

Assume that $n$ is odd. Since $a_{ii}=0$ for all $i$, and each directed walk of length 2 consists of either two positive arcs or either two negative arcs, all nonzero off-diagonal entries of $A^2$ are $+$. Further,
$$(A^2)_{ij}=\begin{cases}
+, & \mbox{if } j=i+2 \mod n; \\
0, & \mbox{otherwise}.
\end{cases}$$
So $A^2$ is irreducible and all nonzero off-diagonal entries are $+$. Therefore $A^2$ requires algebraic positivity, by Lemma \ref{lem3.9}, and hence $A$ requires algebraic positivity, by Lemma \ref{lem3.7}.

Assume that $n$ is even and $A$ is SNS. Then by Lemma \ref{lem3.13}, $A$ requires algebraic positivity.
\end{proof}

\subsection{Cycles with both positive and negative 2-cycles}
In this paper, we say that a sign pattern matrix of size $n\times(n+1)$ is in standard form if it has the following pattern:
\begin{equation}\label{eq3.2}
\begin{bmatrix}
-&+&0&\cdots&0 \\
-&-&+&\ddots&\vdots \\
\vdots&&\ddots&\ddots&0 \\
-&\cdots&\cdots&-&+
\end{bmatrix}
\end{equation}
Let us consider the following row/column transformations:
\begin{enumerate}
\item[] Transformation $T_1$: Interchange any two rows.
\item[] Transformation $T_2$: Multiply a row by $-$.
\item[] Transformation $T_3$: Interchange any two columns.
\end{enumerate}
A sign pattern matrix $A$ is a subpattern of another sign pattern matrix $B$ if $A$ is obtained from by replacing some (possibly none) nonzero entries with zero. The following lemma is deduced from a result by Lancaster \cite{L62}.

\begin{lemma}\label{lem3.16}
Let $A$ be a square matrix and every row has a positive and a negative entry. If for some $i$, the pattern of $A(i,:)$ can be transformed into a subpattern of the standard form \eqref{eq3.2} under transformations $T_1,T_2,T_3$, then $A\bar{x}=\bar{0}$ has a solution with all $x_i$ positive.
\end{lemma}

\begin{proof}
If $A\bar{x}=\bar{0}$ for some $\bar{x}\neq\bar{0}$, then the signs for $x_i/x_j$ do not change under transformations $T_1,T_2,T_3$. So we can assume that $A(i,:)$ is a subpattern of the standard form. Since every row of $A$ has a positive and a negative entry, it is clear that $A\bar{x}=\bar{0}$ has a solution with all $x_i$ positive.
\end{proof}

We consider Type-I and Type-II patterns respectively as follows:
$$\begin{bmatrix}
-_0 & + & 0 & \cdots & 0 \\
- & -_0 & + & \ddots & \vdots \\
0 & - & -_0 & \ddots & 0 \\
\vdots & \ddots & \ddots & \ddots & + \\
0 & \cdots & 0 & - & -_0
\end{bmatrix} \qquad \mbox{and} \qquad 
\begin{bmatrix}
\# & + & 0 & \cdots & 0 \\
+ & \# & + & \ddots & \vdots \\
0 & + & \# & \ddots & 0 \\
\vdots & \ddots & \ddots & \ddots & + \\
0 & \cdots & 0 & + & \#
\end{bmatrix}.$$

\begin{lemma}\label{lem3.17}
If $B\bar{x}=\bar{0}$ for some $B$ of Type-II with non-negative diagonal and a positive vector $\bar{x}$, and $\alpha>0$ such that $B-\alpha I$ has negative diagonal, then there exists a positive vector $\bar{y}$ such that the last entry of $(B-\alpha I)\bar{y}$ is negative and all other entries are 0.
\end{lemma}

\begin{proof}
We may assume that $B$ has the form
$$B=\begin{bmatrix}
d_1 & b_1 & 0 & \cdots & 0 \\
a_1 & d_2 & b_2 & \ddots & \vdots \\
0 & a_2 & d_3 & \ddots & 0 \\
\vdots & \ddots & \ddots & \ddots & b_{n-1} \\
0 & \cdots & 0 & a_{n-1} & d_n
\end{bmatrix}$$
such that all $a_i,b_i$ are positive, and all $d_i$ are non-negative. Further, $d_i-\alpha<0$ for all $i$.

Consider $y_1=x_1$. Then we have
\begin{align*}
d_1x_1+b_1x_2=0 \quad \mbox{and} \quad (d_1-\alpha)y_1+b_1y_2=0 \Rightarrow & y_2=x_2+\frac{\alpha x_1}{b_1} \\
\Rightarrow & y_2=x_2+\epsilon_2, \hspace*{2 cm} \epsilon_2>0; \\[2 mm]
a_1x_1+d_2x_2+b_2x_3=0 \quad \mbox{and} \quad a_1y_1+(d_2-\alpha)y_2+b_2y_3=0 \Rightarrow & -d_2x_2-b_2x_3+(d_2-\alpha)(x_2+\epsilon_2)+b_2y_3=0 \\
\Rightarrow & y_3=x_3+\frac{\alpha x_2}{b_2}-\frac{(d_2-\alpha)\epsilon_2}{b_2} \\
\Rightarrow & y_3=x_3+\epsilon_3, \hspace*{2 cm} \epsilon_3>0; \\[2 mm]
a_2x_2+d_3x_3+b_3x_4=0 \quad \mbox{and} \quad a_2y_2+(d_3-\alpha)y_3+b_3y_4=0 \Rightarrow & -d_3x_3\left(1+\frac{\epsilon_2}{x_2}\right)-b_3x_4\left(1+\frac{\epsilon_2}{x_2}\right) \\
& +(d_3-\alpha)(x_3+\epsilon_3)+b_3y_4=0 \\
\Rightarrow & y_4=x_4\left(1+\frac{\epsilon_2}{x_2}\right)+\frac{d_3x_3\epsilon_2}{b_3x_2}+\frac{\alpha x_3}{b_3}-\frac{(d_3-\alpha)\epsilon_3}{b_3} \\
\Rightarrow & y_4=x_4+\epsilon_4, \hspace*{2 cm} \epsilon_4>0; \\[2 mm]
\vdots & \\[2 mm]
& y_n=x_n+\epsilon_n, \hspace*{2 cm} \epsilon_n>0.
\end{align*}
Since $a_{n-1}x_{n-1}+d_nx_n=0$,
\begin{align*}
a_{n-1}y_{n-1}+(d_n-\alpha)y_n=-d_nx_n\left(1+\frac{\epsilon_{n-1}}{x_{n-1}}\right)+(d_n-\alpha)(x_n+\epsilon_n)=-\frac{d_nx_n\epsilon_{n-1}}{x_{n-1}}-\alpha x_n+(d_n-\alpha)\epsilon_n<0.
\end{align*}
Hence the result follows.
\end{proof}

The following result characterizes sign pattern matrices with cycle graphs and both positive and negative 2-cycles that require algebraic positivity.

\begin{theorem}\label{thm3.18}
Let $A$ be a sign pattern matrix of the form \eqref{cycle} such that $D(A)$ contains both positive and negative 2-cycles. Then $A$ requires algebraic positivity if and only if the following conditions hold:
\begin{enumerate}
\item $a_{1n}=0$ and $a_{ij}=+$ for $j=i+1 \mod n$.
\item Either $A$ has a principal submatrix of the form
$$\begin{bmatrix}
+_0&+ \\ +&\#
\end{bmatrix} \quad \mbox{or} \quad 
\begin{bmatrix}
\#&+ \\ +&+_0
\end{bmatrix},$$
or $A$ has a principal submatrix of the form
$$\begin{bmatrix}
-_0 & + & 0 & \cdots & 0 \\
- & -_0 & + & \ddots & \vdots \\
0 & - & -_0 & \ddots & 0 \\
\vdots & \ddots & \ddots & \ddots & + \\
0 & \cdots & 0 & - & -_0
\end{bmatrix}$$
requiring singularity such that it is not a principal submatrix of any other principal submatrix of $A$ having the same form.
\end{enumerate}
\end{theorem}

\begin{proof}
We assume that $A$ is a sign pattern matrix of the form \eqref{cycle} such that $D(A)$ contains both positive and negative 2-cycles.

\paragraph{To show that the conditions are necessary:} Suppose that $A$ requires algebraic positivity. Then using Lemma \ref{lem3.3}, we can conclude that $a_{1n}=0$ and $a_{ij}=+$ for $j=i+1 \mod n$. So we can partition $A$ in the form
\begin{equation}\label{eq3.3}
A=\begin{bmatrix}
A_{11} & A_{12} & O & \cdots & O \\
O & A_{22} & A_{23} & \ddots & \vdots \\
\vdots & \ddots & \ddots & \ddots & O \\
O & \cdots & O & \ddots & A_{k-1,k} \\
A_{k1} & O & \cdots & O & A_{kk}
\end{bmatrix},
\end{equation}
where each $A_{ii}$ is of either Type-I or Type-II. 

Suppose that conditions 1 and 2 are not satisfied. Then the diagonal entries of all $A_{ii}$ of Type-II are $-$ and no $A_{ii}$ of Type-I requires singularity. So each $A_{ii}$ is potentially stable (see \cite[Theorem 10.2.2]{BS95}). Therefore $A$ is potentially stable.

Without loss of generality, we may assume that $A_{11}$ is of Type-I. Suppose that $A_{11}$ is of order $r$. By Lemma \ref{lem3.4}, $a_{11}=-_0$ and $a_{rr}=-_0$.
\begin{enumerate}
\item If either $a_{11}=0$ or $a_{rr}=0$, then by Proposition \ref{prop2.7}, $A$ does not require algebraic positivity, a contradiction.
\item If $a_{11}=-$ and $a_{rr}=-$, then we can choose $C\in Q(A)$ and $\alpha\geq 0$ such that each column of $(C+\alpha I)(1,:)$ has both positive and negative entries. So by Proposition \ref{prop2.3}, $A$ does not require algebraic positivity, a contradiction. 
\end{enumerate} 
Therefore conditions 1 and 2 are necessary.

\paragraph{To show that the conditions are sufficient:}
Since $a_{1n}=0$ and $a_{ij}=+$ for $j=i+1 \mod n$, we can partition $A$ as in \eqref{eq3.3} so that each of $A_{12},A_{23},\ldots,A_{k-1,k},A_{k1}$ has exactly one nonzero entry $+$ at the bottom left corner of that matrix.

Let $C\in Q(A)$, and it be partitioned as $A$. If $p(x),p_1(x),p_2(x),\ldots,p_k(x)$ are the characteristic polynomials of $C,C_{11},C_{22},\ldots,C_{kk}$, respectively, then $p(x)=p_1(x)p_2(x)\cdots p_k(x)-t$, where $t=c_{12}c_{23}\cdots c_{n-1,n}c_{n1}>0$.

\paragraph{Case-I.} $A$ has a principal submatrix of the form
$$\begin{bmatrix}
+_0&+ \\ +&\#
\end{bmatrix} \quad \mbox{or} \quad 
\begin{bmatrix}
\#&+ \\ +&+_0
\end{bmatrix}.$$ 
Because of Lemma \ref{lem2.6}.4, we can assume that some $C_{ii}$ has a 2-by-2 principal submatrix with pattern 
$$\begin{bmatrix}
0&+ \\ +&\#
\end{bmatrix} \quad \mbox{or} \quad 
\begin{bmatrix}
\#&+ \\ +&0
\end{bmatrix}.$$
Each matrix with these patterns has a positive eigenvalue. Further, each $C_{ii}$ of Type-II is similar to a symmetric matrix. So by Cauchy's interlacing theorem, some $C_{ii}$ has a positive eigenvalue, say $\mu$. Without loss of generality, we can assume that $\mu$ is the largest among eigenvalues of all $C_{ii}$ of Type-II.

For each $C_{ii}$ of Type-II, there exists $\alpha_i$ such that $C_{ii}+\alpha_iI$ is an irreducible nonnegative matrix. So by using the Perron-Frobenius theorem, we can deduce that each $C_{ii}$ of Type-II has left and right positive eigenvectors corresponding to its largest eigenvalue.

Now $p(\mu)=-t<0$. So by intermediate value property, $C$ has an eigenvalue $\lambda>\mu$ so that all diagonal entries of $C-\lambda I$ are negative.

Let $C\bar{x}=\lambda\bar{x}$ for some nonzero $\bar{x}$. Since each of $C_{12},C_{23},\ldots,C_{k-1,k},C_{k1}$ has exactly one nonzero entry, which is positive, we can make following assertions:
\begin{enumerate}
\item Due to Lemma \ref{lem3.17}, we can choose all $x_i$ corresponding to columns of $C_{ii}$ of Type-II to be positive.
\item After removing $n$-th row, all rows corresponding to $C_{ii}$ of Type-II and all zero columns, we have a block diagonal matrix and each diagonal block is a subpattern of the standard form \eqref{eq3.2} such that each row contains a positive and a negative entry. So by Lemma \ref{lem3.16}, all $x_i$ corresponding to columns of $C_{ii}$ of Type-I are positive.  
\end{enumerate} 
Therefore $C$ has a positive right eigenvector corresponding to the eigenvalue $\lambda$.

Using similar arguments on $C^T$, it can be shown that $C$ has a positive left eigenvector corresponding to the eigenvalue $\lambda$.

Since $C$ has the form \eqref{cycle} such that $c_{1n}=0$ and $c_{ij}>0$ for $j=i+1 \mod n$, geometric multiplicity of the eigenvalue $\lambda$ is 1. So by Lemma \ref{lem2.1}.2, $\lambda$ is a simple eigenvalue of $C$. Therefore by Theorem \ref{thm1.1}, $C$ is algebraically positive, and hence $A$ requires algebraic positivity.

\paragraph{Case-II.} $A$ has a principal submatrix of the form
$$\begin{bmatrix}
-_0 & + & 0 & \cdots & 0 \\
- & -_0 & + & \ddots & \vdots \\
0 & - & -_0 & \ddots & 0 \\
\vdots & \ddots & \ddots & \ddots & + \\
0 & \cdots & 0 & - & -_0
\end{bmatrix}$$
requiring singularity such that it is not a principal submatrix of any other principal submatrix of $A$ having the same form. 

If the condition in Case-I does not hold, then all $C_{ii}$ of Type-II has negative diagonal. According to the given condition, we have one $C_{ii}$ of Type-I having eigenvalue 0. Then $p(0)=-t<0$. So by intermediate value property, $C$ has an eigenvalue $\lambda>0$. Without loss of generality, we may assume that $\lambda$ is larger than all eigenvalues of all $C_{ii}$ so that all diagonal entries of $C-\lambda I$ are negative. Therefore using similar arguments as in Case-I it can be shown that $A$ requires algebraic positivity.
\end{proof}

\section{Summary}\label{sec4}
As we have considered that $G(A)$ is a cycle, there are no directed $k$-cycles in $D(A)$ for $2<k<n$ and $A$ is of the form \eqref{cycle}. Lemma \ref{lem3.2} shows that a sign pattern matrix of the form \eqref{cycle} must have a directed $n$-cycle with positive arcs only for $A$ to require algebraic positivity. If all directed 2-cycles are positive, then $A$ is irreducible and all nonzero off-diagonal entries are positive so that $A$ requires algebraic positivity. If there are no loops, then sign patterns requiring algebraic positivity is given by Theorem \ref{thm3.10}. If there are some loops and all directed 2-cycles are negative, then the characterization of sign pattern requiring algebraic positivity is described by Theorem \ref{thm3.15}. If we have both positive and negative directed 2-cycles, then the characterization for $A$ to require algebraic positivity is described by Theorem \ref{thm3.18}. Thus Theorem \ref{thm3.10}, \ref{thm3.15} and \ref{thm3.18} together describes all sign pattern matrices with cycle graphs that require algebraic positivity.

\section*{Declaration of competing interest}
There is no competing interest associated with this work.

\section*{Data availability}
No data was used for the research described in the article.


\begin{thebibliography}{15}

\bibitem{AP19}
J.~Abagat and D.~Pelejo.
\newblock On sign pattern matrices that allow or require algebraic positivity.
\newblock \emph{Electron. J. Linear Algebra}, 35: 331--356, 2019.

\bibitem{BS95}
R.A.~Brualdi and B.L.~Shader.
\newblock Matrices of sign-solvable linear systems.
\newblock Cambridge University Press, Cambridge, 1995.

\bibitem{D22}
S.~Das .
\newblock Sign patterns that allow algebraic positivity.
\newblock \emph{Linear Algebra Appl.}, 653: 91--122, 2022.

\bibitem{DB19}
S.~Das and S.~Bandopadhyay.
\newblock On some sign patterns of algebraically positive matrices.
\newblock \emph{Linear Algebra Appl.}, 562: 91--122, 2019.

\bibitem{HJ13}
R.~A. Horn and C.~R. Johnson.
\newblock \emph{Matrix analysis}.
\newblock Cambridge University Press, Cambridge, Second edition, 2013.

\bibitem{HL14} 
F J. Hall and Z. Li.
\newblock {\em Sign pattern matrices}, in: Handbook of linear algebra, Edited by Leslie Hogben.
\newblock Second edition, CRC Press, Boca Raton, FL, 2014.

\bibitem{KQZ16}
S.~Kirkland, P.~Qiao, and X.~Zhan.
\newblock Algebraically positive matrices.
\newblock \emph{Linear Algebra Appl.}, 504: 14--26, 2016.

\bibitem{L62}
K.~Lancaster.
\newblock The scope of qualitative economics.
\newblock \emph{Rev. Econ. Stud.}, 29: 99-123, 1962.

\bibitem{M00}
C. R.~MacCluer.
\newblock The many proofs and applications of Perron's theorem.
\newblock \emph{SIAM Review}, 42(3): 487--498, 2000.

\end{thebibliography}
\end{document}